\documentclass{amsproc}

\newtheorem{thm}{Theorem}[section]
\newtheorem{lem}[thm]{Lemma}
\newtheorem{cor}[thm]{Corollary}
\newtheorem{prop}[thm]{Proposition}

\theoremstyle{definition}
\newtheorem{defn}[thm]{Definition}

\newtheorem{notn}[thm]{Notation}

\theoremstyle{remark}
\newtheorem{rem}[thm]{Remark}

\numberwithin{equation}{section}

\newcommand{\g}{\mathfrak{g}}
\newcommand\C[1]{\mathcal{#1}}
\newcommand{\spec}{{\rm Spec}}
\newcommand{\prim}{{\rm Prim}}
\newcommand{\gc}{ [ \hspace{-0.65mm} [}
\newcommand{\dc}{]  \hspace{-0.65mm} ]}

\usepackage[pdftex]{graphicx}

\usepackage{pgf}

\begin{document}

\title[Enumeration of torus-invariant strata]{Enumeration of torus-invariant strata with respect to dimension in the big cell of the quantum minuscule Grassmannian of type $\mathbf{B_n}$}

%    Information for first author
\author{Jason Bell}
%    Address of record for the research reported here
\address{Department of Mathematics, Simon Fraser University, Burnaby, BC V5A 1S6, Canada}
%    Current address
%\curraddr{Department of Mathematics and Statistics,
%Case Western Reserve University, Cleveland, Ohio 43403}
\email{jpb@math.sfu.ca}
%    \thanks will become a 1st page footnote.
\thanks{The first and second named authors thank NSERC for its generous support.}

%    Information for second author
\author{Karel Casteels}
\address{Department of Mathematics, University of California, Santa Barbara, CA 93106, USA}
\email{casteels@math.ucsb.edu}
%\thanks{The second named authors thank NSERC for its generous support.}

\author{St\'ephane Launois}
\address{School of Mathematics, Statistics and Actuarial Science, University of Kent, Canterbury, Kent CT2 7NF, United Kingdom}
\email{S.Launois@kent.ac.uk}
%\thanks{The second named authors thank NSERC for its generous support.}

%    General info
\subjclass[2000]{16W35; 20G42}\date{December 31, 2010 and, in revised form, XXX.}
\dedicatory{This paper is dedicated to Ken Goodearl for his 65th birthday.}

\keywords{Quantum algebras, primitive ideals, algebraic combinatorics}

\begin{abstract}
The aim of this article is to give explicit formulae for various generating functions, including the generating function of torus-invariant primitive ideals in the big cell of the quantum minuscule grassmannian of type $B_n$.
\end{abstract}

\maketitle

%\section*{This is an unnumbered first-level section head} This is an example of an unnumbered first-level heading.

%\specialsection*{This is a Special Section Head} This is an example of a special section head%
%%%%%%%%%%%%%%%%%%%%%%%%%%%%%%%%%%%%%%%%%%%%%%%%%%%%%%%%%%%%%%%%%%%%%%%%
%\footnote{Here is an example of a footnote. Notice that this footnote text is running on so that it can stand as an example of how a footnote
%with separate paragraphs should be written.\parAnd here is the beginning of the second paragraph.}%
%%%%%%%%%%%%%%%%%%%%%%%%%%%%%%%%%%%%%%%%%%%%%%%%%%%%%%%%%%%%%%%%%%%%%%%%

\section{Introduction}

Let $\g$ be a simple Lie algebra of rank $n$ over the field of complex numbers, and let $\pi:=\{\alpha_1,\dots,\alpha_n\}$ be the set of simple roots associated to a triangular decomposition $\g=\mathfrak{n}^- \oplus \mathfrak{h} \oplus \mathfrak{n}^+$. Let $W$ be the Weyl group associated to $\g$. 

The aim of this article is to study the prime spectrum of so-called quantum Schubert cells from the point of view of algebraic combinatorics. Quantum Schubert cells have been introduced by De Concini-Kac-Procesi as quantisations of enveloping algebras of nilpotent Lie algebras $\mathfrak{n}_w:= \mathfrak{n}^+ \cap \mathrm{Ad}_{w}(\mathfrak{n}^-) $, where $\mathrm{Ad}$ stands for the adjoint action and $w \in W$.  These noncommutative algebras are defined thanks to the braid group action of $W$ on the quantised enveloping algebra $U_q(\g)$ induced by Lusztig automorphisms. The resulting (quantum) algebra associated to a chosen $w \in W$ is denoted by $U_q[w]$. Here $q$ denotes a nonzero element of the base field $\mathbb{K}$, and we assume that $q$ is not a root of unity. It was recently shown by Yakimov that these algebras can be seen as the Schubert cells of the quantum flag varieties. Our aim is to study combinatorially the prime spectrum of the algebras $U_q[w]$. In order to explain the main results of this paper, a brief sketch of background is needed.

\subsection{$\C{H}$-Stratification.} In order to investigate the primitive ideals of various quantum algebras, Goodearl and Letzter have developed a strategy based on the rational action of a torus. More precisely, they define a stratification of the prime and primitive spectra of an algebra $A$ supporting a rational action of a torus $H$ \cite{gletDM}. In the context of the quantum Schubert cell $U_q[w]$, there is a natural action of the torus $\C{H}:=(\mathbb{K}^*)^n$, and the associated stratification of the prime spectrum is parametrised by those prime ideals that are invariant under this torus action, the so-called $\C{H}$-primes. Moreover each stratum is homeomorphic to the prime spectrum of a commutative Laurent polynomial ring over $\mathbb{K}$. Torus-invariant prime ideals of $U_q[w]$ have recently been studied by M\'eriaux and Cauchon \cite{mer1} on one hand and by Yakimov \cite{yakimov} on the other hand.  In particular, they proved that $\C{H}$-invariant primes in $U_q[w]$ are in one-to-one correspondence with the initial Bruhat interval $[\mathrm{id},w]$. Hence the stratification of the prime spectrum of $U_q[w]$ can be written:
\begin{eqnarray}
\label{eq:stratification}
\spec(U_q[w]) =\bigsqcup_{v \leq w} \spec_v(U_q[w]),
\end{eqnarray}
where the stratum $ \spec_v(U_q[w])$ associated to $v \leq w$  is homeomorphic to the prime spectrum $$\spec \left( \mathbb{K}[z_{v,1}^{\pm 1}, \dots , z_{v,d(v)}^{\pm 1}] \right)$$ of a commutative Laurent polynomial ring over the base field. The (Krull) dimension of these commutative Laurent polynomial rings has recently been computed in \cite{BCLuqw}, where we prove that 
\begin{eqnarray}
\label{eqdimstratum}
\dim  \spec_v(U_q[w]) =d(v)= \dim(\ker(v+w)),
\end{eqnarray}
$v$ and $w$ acting on the dual $\mathfrak{h}^*$ of the Cartan subalgebra $\mathfrak{h}$. See also \cite{Yak3} where a similar formula is established but with stronger hypothesis on $\mathbb{K}$ and $q$.

\subsection{Big cells of quantum Grassmannians.} Let $j$ be an element of the set $[n]:=\{1, \dots , n\}$ and  $J:=\{1, \dots, n\} \setminus \{j\}$. We denote by $W_J$ the associated parabolic subgroup. Recall that $W_J$ is the subgroup of $W$ generated by $s_i$ with $i \in J$. We denote by $W^J$ the set of minimal length coset representatives of $W/W_J$. It has a unique maximal element (for the induced Bruhat order) that we denote by $w^J_{max}$. 

As explained above, Yakimov \cite{yakimovflag} recently proved that the algebra $U_q[w^J_{max}]$ can actually be seen as the big cell of the quantum Grassmannian associated to the fundamental weight $\varpi_j$. More precisely, there exists an isomorphism between a localisation of the quantum Grassmannian associated to the  the fundamental weight $\varpi_j$ and a skew-Laurent extension of the algebra $U_q[w^J_{max}]$. This isomorphism allows one to transfer information between these two algebras, and shows that to understand the prime spectrum of the quantum Grassmannian associated to the fundamental weight $\varpi_j$ a first step is to classify the prime spectrum of the algebra $U_q[w^J_{max}]$.

For example, the well-known algebra of $m\times n$ quantum matrices appears in this context as it was proved by M\'eriaux and Cauchon that this algebra is isomorphic to 
$U_q[w^{\{1,\dots,m+n-1\}\setminus\{m\}}_{max}]$ when $\g$ is of type $A_{m+n-1}$. In \cite{BCLqm}, we used our formula (\ref{eqdimstratum}) for the dimension of a stratum in order to derive various enumeration results. In particular, we give a closed formula for the trivariate generating function that counts the $d$-dimensional $\C{H}$-strata in $m\times n$ quantum matrices.

\subsection{Main results.} In this paper, we consider the case where the Lie algebra $\g$ is of type $B_n$ and $\varpi_j$ is the unique minuscule weight. That is, we let $\g$ be the simple Lie algebra $\mathfrak{so}_{2n+1}$ over the field of complex numbers, and let $\pi:=\{\alpha_1,\dots,\alpha_n\}$ be the set of simple roots, where $\alpha_n$ is the unique short simple root. 
Moreover, we set $J:=\{1, \dots, n-1\}$, and $W_J$ denotes the associated parabolic subgroup. The unique maximal element of $W/W^J$ is  denoted by $w^J_{max}$. The main result of this article gives an explicit formula for the two term generating function $H(x,t)$ whose coefficient of $\frac{x^n}{n!}t^d$ is the number of $d$-dimensional $\C{H}$-strata in $U_q[w^J_{max}]$. More precisely, we prove the following result.
\begin{thm}
\label{thmintro}
Let $\g$ be a simple Lie algebra of type $B_n$, and let $w^J_{max}$ be the unique minimal length coset representative of $W/W_{\{1,\dots ,n-1\}}$. 
Denote by  $H(x,t)$ the two term generating function whose coefficient of $\frac{x^n}{n!}t^d$ is the number of $d$-dimensional $\C{H}$-strata in $U_q[w^J_{max}]$.
Then 
$$H(x,t)=\left(\frac{e^x}{2-e^x}\right)^{\frac{t+1}{2}}.$$
\end{thm}

Specialising at $t=1$, we obtain the exponential generating function for the number of $\C{H}$-strata in $U_q[w^J_{max}]$. On the other hand, specialising at $t=0$, we obtain the exponential generating function for the number of $\C{H}$-strata of dimension 0 in $U_q[w^J_{max}]$. By the Stratification Theorem of Goodearl and Letzter (see for instance \cite{bg}), a stratum is 0-dimensional exactly when the associated $\C{H}$-prime is primitive. So we deduce from the previous theorem that the exponential generating function for the primitive $\C{H}$-primes in $U_q[w^J_{max}]$ is 
$$\left(\frac{e^x}{2-e^x}\right)^{\frac{1}{2}}.$$

Finally, we use Theorem \ref{thmintro} in order to prove that the proportion of primitive $\C{H}$-primes tends to $0$ as $n\to \infty$.

The methods developed in order to prove the above theorem are somehow similar to those developed in \cite{BCLqm} to attack the quantum matrix case. 
In particular, we will use the notion of pipe dreams in the context of ``signed permutations'', and develop the notion of symmetric Cauchon diagrams which can be seen as a type $B_n$ analogue of the well-known Cauchon diagrams for quantum matrices. Symmetric Cauchon diagrams have roughly speaking already appeared under the name \reflectbox{L}-diagrams of type ($B_n$,$n$) in \cite{LW}. \\

Throughout this paper, we use the following conventions. 
\begin{enumerate}
\item[$(i)$]
 If $R$ is a finite set or sequence, $|R|$ denotes its cardinality.  
\item[$(ii)$]  For any natural number $t$, we set $[t]:=\{1, \dots , t\}=\gc 1,t \dc$. 
\item[$(iii)$] $\mathbb{K}$ denotes an infinite field and we set
$\mathbb{K}^*:=\mathbb{K}\setminus \{0\}$. 
\item[$(iv)$] $q \in \mathbb{K}^*$ is not a root of unity. 
\item[$(v)$] If $A$ is a $\mathbb{K}$-algebra, then $\spec(A)$ and $\prim(A)$ denote respectively its prime and primitive spectra.\\
\end{enumerate}

\noindent {\bf Acknowledgments:} We thank the anonymous referee for comments that have greatly improved this text.

\section{Cauchon diagrams and Permutations in type $\mathbf{B_n}$}

\subsection{Cauchon diagrams}

Consider any $w\in W$, and set $t := l(w)$. Let $w= s_{i_1} \circ \cdots \circ s_{i_t}$ $(i_j \in \{1, \dots , n\})$ be a reduced decomposition of $w$. It is well known that
$\beta_1 = \alpha_{i_1}$, $\beta_2 = s_{i_1}(\alpha_{i_2})$, $\ldots$, $\beta_t = s_{i_1} \circ \cdots \circ s_{i_{t-1}}(\alpha_{i_t})$ 
are distinct positive roots and that the set $\{\beta_1, ..., \beta_t\}$ does not depend on the chosen reduced expression of $w$. 

A subset $\Delta$ of $[t]$ is called a {\it diagram} of $w$. By abuse of notation, we say that $\beta_i$ belongs to $\Delta$ if $i \in \Delta$. 

\begin{notn} {\rm 
Fix a reduced decomposition $w= s_{i_1} \circ \cdots \circ s_{i_t}$ $(i_j \in \{1, \dots , n\})$ of $w$ and a diagram $\Delta\subseteq\{1,2,\ldots,t\}$. 
\begin{enumerate}
\item For all $k\in [t]$, we set \[s_{i_k}^{\Delta}:=\left\{ \begin{array}{ll} s_{i_k} & {\rm if~}k\in \Delta, \\{\rm id} & {\rm otherwise.}\end{array}\right. \]
\item We also set  $$w^\Delta:=s_{i_1}^\Delta\cdots s_{i_t}^\Delta\in W,$$
$$v^\Delta:=s_{i_t}^\Delta\cdots s_{i_1}^\Delta\in W$$
and 
$$v_{\ell}^\Delta:=s_{i_t}^\Delta\cdots s_{i_{t-\ell+1}}^\Delta\in W$$
for all $\ell \in \{0, \dots , t\}$.
\end{enumerate}}
\end{notn}

\begin{defn} \label{cauchondiagramdef}
{\rm A diagram $\Delta$ is {\it Cauchon (with respect to the fixed reduced decomposition of $w$)} if 
$$v_{\ell-1}^{\Delta} < v_{\ell-1}^{\Delta} s_{i_{t-\ell+1}} $$
for all  $\ell \in \{1, \dots , t\}$.}
\end{defn}

In \cite{mer1}, Cauchon diagrams are called \emph{admissible diagrams.} It was proved by M\'eriaux and Cauchon \cite{mer1} that they coincide 
with positive distinguished subexpressions of the reduced decomposition of $w$ in the sense of Marsh and Rietsch \cite{MR}.

The relevance of Cauchon diagrams for this work is summarized in the following theorem. (The reader is referred to \cite{BCLuqw} for details about the algebra $U_q[w]$.)\\

\begin{thm}[\cite{mer1,BCLuqw}]~ \label{thm: bcl}
\begin{enumerate}
\item The map $\Delta \mapsto w^{\Delta}$ is a bijection from the set of Cauchon diagrams onto the Bruhat interval $[\mathrm{id},w]$. 
\item The deleting-derivation algorithm provides a bijection between the set of torus invariant prime ideals in $U_q[w]$ and the set of Cauchon diagrams.
\item Let $J$ be the unique torus-invariant prime ideal  of $U_q[w]$ associated to the Cauchon diagram $\Delta$. Then the dimension of the associated $\mathcal{H}$-stratum is equal to $\dim(\ker(w^\Delta+w))$.
\end{enumerate}
\end{thm}

In practice, $\dim(\ker(w^\Delta+w))$ is easy to compute using combinatorial tools. For example, in the quantum matrices case, the reader is referred to~\cite{BCLqm} for results that use the above formula for the dimension of an $\mathcal{H}$-stratum in a crucial way. The present work will also make use of this formula.

\subsection{The big cell of the minuscule quantum Grassmannian of type $\mathbf{B_n}$}

From now on, $\g$ is the Lie algebra $\mathfrak{so}_{2n+1}(\mathbb{C})$. Its simple roots $\alpha_1, \dots , \alpha_n$ are labelled so that $\alpha_n$ is the unique short root. 

Moreover, from now on, we set: 
$$w=w_{max}^{\{1,\dots,n-1\}}=s_n (s_{n-1} s_n) \cdots (s_1s_2\cdots s_{n-1} s_n).$$
Set $J=\{1, \dots ,n-1\}$ and let $W_J$ denote the subgroup of $W$ generated by the $s_i$ with $i \neq n$. Then 
$w$ is the unique minimal length representative of the coset $w_0+W_J$, where $w_0$ is the longest element in $W$.  
In other words, $w$ is the element of maximal length in $W^J$.

The chosen reduced decomposition of $w$ can be represented by the following Young tableau:
 \[\begin{array}{c|c|c|c|c|c|}
\cline{2-2}
&s_n& \multicolumn{4}{c}{} \\
\cline{2-3}
 &s_{n-1}& s_{n}& \multicolumn{3}{c}{}  \\
 \cline{2-4}
 & & \cdots & &\multicolumn{2}{c}{}   \\
 \cline{2-5}
 & s_{1} & \dots & \dots & s_{n} & \multicolumn{1}{c}{}  \\
\cline{2-5}
\end{array} \]

The positive roots $\beta_1$, ..., $\beta_{\frac{n(n+1)}{2}}$ can be placed in a Young tableau shape as follows:
 
 \[\begin{array}{c|c|c|c|c|c|}
\cline{2-2}
&\beta_1& \multicolumn{4}{c}{} \\
 & &  \multicolumn{4}{c}{} \\
\cline{2-3}
 &\beta_2& \beta_3& \multicolumn{3}{c}{}  \\
  & & &  \multicolumn{3}{c}{} \\
 \cline{2-4}
 & & \cdots & &\multicolumn{2}{c}{}   \\
  & & & &  \multicolumn{2}{c}{} \\
 \cline{2-5}
 & \beta_{\frac{n^2-n+2}{2}} & \dots & \dots & \beta_{\frac{n(n+1)}{2}} & \multicolumn{1}{c}{}  \\
 & & & & & \multicolumn{1}{c}{} \\
\cline{2-5}
\end{array} \]

In this case, we can represent a Cauchon diagram $\Delta$  by a certain filling of the Young tableau above 
with black and white, the black boxes indicating the positive roots that belong to $\Delta$. (The box labelled $\beta_i$ is black if and only if $i \in \Delta$.)

A filling of the Young tableau above with black and white is called a {\it diagram of type $B_n$}.

\subsection{Symmetric Cauchon Diagrams}

There is a useful visualization of a Cauchon diagram as a certain coloured grid which we now describe. 

First recall that Cauchon diagrams coincide with positive distinguished subexpressions of $w$ by the work of M\'eriaux-Cauchon \cite{mer1}. 
Hence we deduce from \cite{LW} that Cauchon diagrams coincide with \reflectbox{L}-diagrams in the sense of Lam and Williams. Note however that our convention here are slightly different from those of \cite{LW}. Because of this, a reflection along the $x$-axis must be applied when using the results of Lam and Williams. Moreover Lam and Williams use labellings of Young tableaux by zeroes and crosses, whereas here we colour our Young tableaux with black and white. Our black boxes correspond to their zeroes and our white boxes correspond to their crosses.

In \cite[Theorem 6.1]{LW}, Lam and Williams have described the \reflectbox{L}-diagrams for the case that interests us. As \reflectbox{L}-diagrams and Cauchon diagrams coincide by the above discussion, \cite[Theorem 6.1]{LW} also describes Cauchon diagrams of $U_q[w_{max}^{\{1,\dots,n-1\}}]$. 

\begin{thm}[Lam-Williams, M\'eriaux-Cauchon]
A diagram of type $B_n$ is a Cauchon diagram if and only if the following conditions are satisfied:
\begin{enumerate}
\item If there is a black box below a white box (and in the same column), then all boxes on the left and in the same row as the black box must be black. 
\item If a diagonal box is black, then all boxes to its left (and in the same row) must be black.
\end{enumerate}
\end{thm} 

We note that we can rephrase the conditions that appear in the previous theorem in a somehow simpler way. If $\Delta$ is a Cauchon diagram, then we define the {\it associated symmetric Cauchon diagram} $\Delta'$ as the $n \times n$ grid, where the lower triangular part is $\Delta$ and the upper triangular part is the reflection of $\Delta$ along the diagonal. Figure~\ref{CDexamples} provides some examples.

One easily deduces from the previous theorem the following result.

\begin{cor}
A symmetric Cauchon diagram is an $n\times n$ grid of squares, indexed in the usual way by $[n]\times[n]$, where each square is coloured black or white so that the following properties are satisfied. First, if $(i,j)$ is black, then so is $(j,i)$. Secondly, if $(i,j)$ is black, then so is $(h,j)$ for all $h\in[i]$ or so is $(i,h)$ for all $h\in[j]$. Figure~\ref{CDexamples} provides some examples.
\end{cor}

\begin{figure}[htbp]
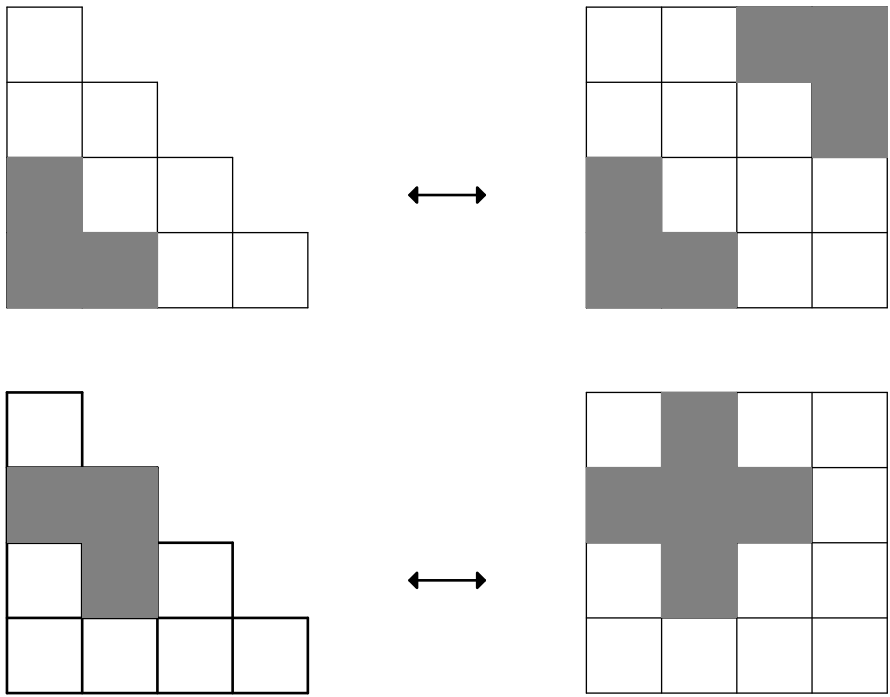


\begin{center}
\begin{tabular}{ccc}
\begin{pgfpicture}{0cm}{0cm}{5cm}{5cm}%
\pgfsetroundjoin \pgfsetroundcap%
\pgfsetlinewidth{0.5pt} 
\pgfxyline(0.5,0.5)(4.5,0.5)
\pgfxyline(0.5,1.5)(4.5,1.5)
\pgfxyline(0.5,2.5)(3.5,2.5)
\pgfxyline(0.5,3.5)(2.5,3.5)
\pgfxyline(0.5,4.5)(1.5,4.5)
\pgfxyline(0.5,0.5)(0.5,4.5)
\pgfxyline(1.5,0.5)(1.5,4.5)
\pgfxyline(2.5,0.5)(2.5,3.5)
\pgfxyline(3.5,0.5)(3.5,2.5)
\pgfxyline(4.5,0.5)(4.5,1.5)
\pgfsetfillcolor{gray}
\pgfmoveto{\pgfxy(0.5,1.5)}\pgflineto{\pgfxy(0.5,2.5)}\pgflineto{\pgfxy(1.5,2.5)}\pgflineto{\pgfxy(1.5,1.5)}\pgflineto{\pgfxy(0.5,1.5)}\pgffill
\pgfmoveto{\pgfxy(0.5,0.5)}\pgflineto{\pgfxy(0.5,1.5)}\pgflineto{\pgfxy(1.5,1.5)}\pgflineto{\pgfxy(1.5,0.5)}\pgflineto{\pgfxy(0.5,0.5)}\pgffill
\pgfmoveto{\pgfxy(1.5,0.5)}\pgflineto{\pgfxy(1.5,1.5)}\pgflineto{\pgfxy(2.5,1.5)}\pgflineto{\pgfxy(2.5,0.5)}\pgflineto{\pgfxy(1.5,0.5)}\pgffill
\end{pgfpicture} 
 & 
\begin{pgfpicture}{0cm}{0cm}{2cm}{4cm}%
\pgfsetroundjoin \pgfsetroundcap%
\pgfsetlinewidth{1pt} 
\pgfxyline(0.5,2)(1.5,2)
\pgfmoveto{\pgfxy(0.5,2)}\pgflineto{\pgfxy(0.6,2.1)}\pgflineto{\pgfxy(0.6,1.9)}\pgflineto{\pgfxy(0.5,2)}\pgfclosepath\pgffillstroke
\pgfmoveto{\pgfxy(1.5,2)}\pgflineto{\pgfxy(1.4,2.1)}\pgflineto{\pgfxy(1.4,1.9)}\pgflineto{\pgfxy(1.5,2)}\pgfclosepath\pgffillstroke
\end{pgfpicture}
 &

\begin{pgfpicture}{0cm}{0cm}{5cm}{5cm}%
\pgfsetroundjoin \pgfsetroundcap%
\pgfsetlinewidth{0.5pt} 
\pgfxyline(0.5,0.5)(4.5,0.5)
\pgfxyline(0.5,1.5)(4.5,1.5)
\pgfxyline(0.5,2.5)(4.5,2.5)
\pgfxyline(0.5,3.5)(4.5,3.5)
\pgfxyline(0.5,4.5)(4.5,4.5)
\pgfxyline(0.5,0.5)(0.5,4.5)
\pgfxyline(1.5,0.5)(1.5,4.5)
\pgfxyline(2.5,0.5)(2.5,4.5)
\pgfxyline(3.5,0.5)(3.5,4.5)
\pgfxyline(4.5,0.5)(4.5,4.5)
\pgfsetfillcolor{gray}
\pgfmoveto{\pgfxy(0.5,1.5)}\pgflineto{\pgfxy(0.5,2.5)}\pgflineto{\pgfxy(1.5,2.5)}\pgflineto{\pgfxy(1.5,1.5)}\pgflineto{\pgfxy(0.5,1.5)}\pgffill
\pgfmoveto{\pgfxy(0.5,0.5)}\pgflineto{\pgfxy(0.5,1.5)}\pgflineto{\pgfxy(1.5,1.5)}\pgflineto{\pgfxy(1.5,0.5)}\pgflineto{\pgfxy(0.5,0.5)}\pgffill
\pgfmoveto{\pgfxy(1.5,0.5)}\pgflineto{\pgfxy(1.5,1.5)}\pgflineto{\pgfxy(2.5,1.5)}\pgflineto{\pgfxy(2.5,0.5)}\pgflineto{\pgfxy(1.5,0.5)}\pgffill
\pgfmoveto{\pgfxy(2.5,3.5)}\pgflineto{\pgfxy(2.5,4.5)}\pgflineto{\pgfxy(3.5,4.5)}\pgflineto{\pgfxy(3.5,3.5)}\pgflineto{\pgfxy(2.5,3.5)}\pgffill
\pgfmoveto{\pgfxy(3.5,3.5)}\pgflineto{\pgfxy(3.5,4.5)}\pgflineto{\pgfxy(4.5,4.5)}\pgflineto{\pgfxy(4.5,3.5)}\pgflineto{\pgfxy(3.5,3.5)}\pgffill
\pgfmoveto{\pgfxy(3.5,2.5)}\pgflineto{\pgfxy(3.5,3.5)}\pgflineto{\pgfxy(4.5,3.5)}\pgflineto{\pgfxy(4.5,2.5)}\pgflineto{\pgfxy(3.5,2.5)}\pgffill
\end{pgfpicture} 
\\
\begin{pgfpicture}{0cm}{0cm}{5cm}{5cm}%
\pgfsetroundjoin \pgfsetroundcap%
\pgfsetlinewidth{1pt} 
\pgfxyline(0.5,0.5)(4.5,0.5)
\pgfxyline(0.5,1.5)(4.5,1.5)
\pgfxyline(0.5,2.5)(3.5,2.5)
\pgfxyline(0.5,3.5)(2.5,3.5)
\pgfxyline(0.5,4.5)(1.5,4.5)
\pgfxyline(0.5,0.5)(0.5,4.5)
\pgfxyline(1.5,0.5)(1.5,4.5)
\pgfxyline(2.5,0.5)(2.5,3.5)
\pgfxyline(3.5,0.5)(3.5,2.5)
\pgfxyline(4.5,0.5)(4.5,1.5)
\pgfsetfillcolor{gray}
\pgfmoveto{\pgfxy(0.5,2.5)}\pgflineto{\pgfxy(0.5,3.5)}\pgflineto{\pgfxy(1.5,3.5)}\pgflineto{\pgfxy(1.5,2.5)}\pgflineto{\pgfxy(0.5,2.5)}\pgffill
\pgfmoveto{\pgfxy(1.5,2.5)}\pgflineto{\pgfxy(1.5,3.5)}\pgflineto{\pgfxy(2.5,3.5)}\pgflineto{\pgfxy(2.5,2.5)}\pgflineto{\pgfxy(1.5,2.5)}\pgffill
\pgfmoveto{\pgfxy(1.5,1.5)}\pgflineto{\pgfxy(1.5,2.5)}\pgflineto{\pgfxy(2.5,2.5)}\pgflineto{\pgfxy(2.5,1.5)}\pgflineto{\pgfxy(1.5,1.5)}\pgffill
\end{pgfpicture} 
 & 
\begin{pgfpicture}{0cm}{0cm}{2cm}{4cm}%
\pgfsetroundjoin \pgfsetroundcap%
\pgfsetlinewidth{1pt} 
\pgfxyline(0.5,2)(1.5,2)
\pgfmoveto{\pgfxy(0.5,2)}\pgflineto{\pgfxy(0.6,2.1)}\pgflineto{\pgfxy(0.6,1.9)}\pgflineto{\pgfxy(0.5,2)}\pgfclosepath\pgffillstroke
\pgfmoveto{\pgfxy(1.5,2)}\pgflineto{\pgfxy(1.4,2.1)}\pgflineto{\pgfxy(1.4,1.9)}\pgflineto{\pgfxy(1.5,2)}\pgfclosepath\pgffillstroke
\end{pgfpicture}
 &

\begin{pgfpicture}{0cm}{0cm}{5cm}{5cm}%
\pgfsetroundjoin \pgfsetroundcap%
\pgfsetlinewidth{0.5pt} 
\pgfxyline(0.5,0.5)(4.5,0.5)
\pgfxyline(0.5,1.5)(4.5,1.5)
\pgfxyline(0.5,2.5)(4.5,2.5)
\pgfxyline(0.5,3.5)(4.5,3.5)
\pgfxyline(0.5,4.5)(4.5,4.5)
\pgfxyline(0.5,0.5)(0.5,4.5)
\pgfxyline(1.5,0.5)(1.5,4.5)
\pgfxyline(2.5,0.5)(2.5,4.5)
\pgfxyline(3.5,0.5)(3.5,4.5)
\pgfxyline(4.5,0.5)(4.5,4.5)
\pgfsetfillcolor{gray}
\pgfmoveto{\pgfxy(0.5,2.5)}\pgflineto{\pgfxy(0.5,3.5)}\pgflineto{\pgfxy(1.5,3.5)}\pgflineto{\pgfxy(1.5,2.5)}\pgflineto{\pgfxy(0.5,2.5)}\pgffill
\pgfmoveto{\pgfxy(1.5,2.5)}\pgflineto{\pgfxy(1.5,3.5)}\pgflineto{\pgfxy(2.5,3.5)}\pgflineto{\pgfxy(2.5,2.5)}\pgflineto{\pgfxy(1.5,2.5)}\pgffill
\pgfmoveto{\pgfxy(1.5,1.5)}\pgflineto{\pgfxy(1.5,2.5)}\pgflineto{\pgfxy(2.5,2.5)}\pgflineto{\pgfxy(2.5,1.5)}\pgflineto{\pgfxy(1.5,1.5)}\pgffill
\pgfmoveto{\pgfxy(1.5,3.5)}\pgflineto{\pgfxy(1.5,4.5)}\pgflineto{\pgfxy(2.5,4.5)}\pgflineto{\pgfxy(2.5,3.5)}\pgflineto{\pgfxy(1.5,3.5)}\pgffill
\pgfmoveto{\pgfxy(2.5,2.5)}\pgflineto{\pgfxy(2.5,3.5)}\pgflineto{\pgfxy(3.5,3.5)}\pgflineto{\pgfxy(3.5,2.5)}\pgflineto{\pgfxy(2.5,2.5)}\pgffill
\end{pgfpicture} 

 \end{tabular}
\end{center}

\caption{Two Cauchon diagrams and their associated symmetric Cauchon diagrams}
\label{CDexamples}

\end{figure}

\begin{notn}
If $\Delta$ is a Cauchon diagram of type $B_n$, we denote the corresponding symmetric Cauchon diagram by $\Delta'$.
\end{notn}

\subsection{Signed Permutations}\label{pipes}
It is well known (e.g., see~\cite{comb}) that for a simple Lie algebra of type $B_n$, the Weyl group $W$ is precisely the group $S^B_{2n}$ of \emph{signed permutations}, that is, permutations $\sigma$ of the set $[n]\cup (-[n])$ where $$\sigma(i)=-\sigma(-i)$$ for all $i\in[n]$. Given a Cauchon diagram $\Delta$ and its associated symmetric Cauchon diagram $\Delta'$, we obtain a signed permutation $\tau_\Delta$ by using the method of \emph{pipe-dreams} that we now describe.  First, label the rows of $\Delta'$ by $1,2,\ldots, n$ (starting from the bottom-most row) and label columns by $-n,-(n-1),\ldots,-1$ (starting from the left-most column). Next place pipes on the squares of $\Delta'$ as follows. Each white square receives an ``oriented pair of elbows'', while each black square receives an ``oriented cross''. The situation for a specific example is pictured in Figure~\ref{pipedreamsexample}.

\begin{figure}[htbp]
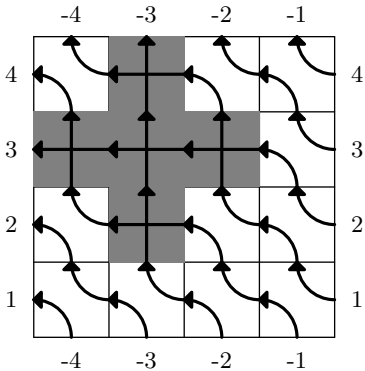

\begin{center}
\begin{pgfpicture}{0cm}{0cm}{5.5cm}{5.5cm}%
\pgfsetroundjoin \pgfsetroundcap%
\pgfsetlinewidth{0.5pt} 
\pgfxyline(0.5,0.5)(4.5,0.5)
\pgfxyline(0.5,1.5)(4.5,1.5)
\pgfxyline(0.5,2.5)(4.5,2.5)
\pgfxyline(0.5,3.5)(4.5,3.5)
\pgfxyline(0.5,4.5)(4.5,4.5)
\pgfxyline(0.5,0.5)(0.5,4.5)
\pgfxyline(1.5,0.5)(1.5,4.5)
\pgfxyline(2.5,0.5)(2.5,4.5)
\pgfxyline(3.5,0.5)(3.5,4.5)
\pgfxyline(4.5,0.5)(4.5,4.5)
\pgfsetfillcolor{gray}
\pgfmoveto{\pgfxy(0.5,2.5)}\pgflineto{\pgfxy(0.5,3.5)}\pgflineto{\pgfxy(1.5,3.5)}\pgflineto{\pgfxy(1.5,2.5)}\pgflineto{\pgfxy(0.5,2.5)}\pgffill
\pgfmoveto{\pgfxy(1.5,2.5)}\pgflineto{\pgfxy(1.5,3.5)}\pgflineto{\pgfxy(2.5,3.5)}\pgflineto{\pgfxy(2.5,2.5)}\pgflineto{\pgfxy(1.5,2.5)}\pgffill
\pgfmoveto{\pgfxy(1.5,1.5)}\pgflineto{\pgfxy(1.5,2.5)}\pgflineto{\pgfxy(2.5,2.5)}\pgflineto{\pgfxy(2.5,1.5)}\pgflineto{\pgfxy(1.5,1.5)}\pgffill
\pgfmoveto{\pgfxy(1.5,3.5)}\pgflineto{\pgfxy(1.5,4.5)}\pgflineto{\pgfxy(2.5,4.5)}\pgflineto{\pgfxy(2.5,3.5)}\pgflineto{\pgfxy(1.5,3.5)}\pgffill
\pgfmoveto{\pgfxy(2.5,2.5)}\pgflineto{\pgfxy(2.5,3.5)}\pgflineto{\pgfxy(3.5,3.5)}\pgflineto{\pgfxy(3.5,2.5)}\pgflineto{\pgfxy(2.5,2.5)}\pgffill

%Box (1,1)
\pgfsetlinewidth{1.2pt} 
\color{black}
\pgfmoveto{\pgfxy(0.5,4)}\pgfpatharc{90}{0}{0.5cm}
\pgfstroke
\pgfmoveto{\pgfxy(0.5,4)}\pgflineto{\pgfxy(0.6,4.1)}\pgflineto{\pgfxy(0.6,3.9)}\pgflineto{\pgfxy(0.5,4)}\pgfclosepath\pgffillstroke
\color{black}
\pgfmoveto{\pgfxy(1,4.5)}\pgfpatharc{180}{270}{0.5cm}
\pgfstroke
\pgfmoveto{\pgfxy(1,4.5)}\pgflineto{\pgfxy(0.9,4.4)}\pgflineto{\pgfxy(1.1,4.4)}\pgflineto{\pgfxy(1,4.5)}\pgfclosepath\pgffillstroke

%Box (1,2)
\pgfsetlinewidth{1.2pt} 
\color{black} 
\pgfxyline(2,3.5)(2,4.5)
\pgfmoveto{\pgfxy(1.9,4.4)}\pgflineto{\pgfxy(2,4.5)}\pgflineto{\pgfxy(2.1,4.4)}\pgflineto{\pgfxy(1.9,4.4)}\pgfclosepath\pgffillstroke
\color{black}
\pgfxyline(1.5,4)(2.5,4)
\pgfmoveto{\pgfxy(1.5,4)}\pgflineto{\pgfxy(1.6,4.1)}\pgflineto{\pgfxy(1.6,3.9)}\pgflineto{\pgfxy(1.5,4)}\pgfclosepath\pgffillstroke

%Box (1,3)
\pgfsetlinewidth{1.2pt} 
\color{black}
\pgfmoveto{\pgfxy(2.5,4)}\pgfpatharc{90}{0}{0.5cm}
\pgfstroke
\pgfmoveto{\pgfxy(2.5,4)}\pgflineto{\pgfxy(2.6,4.1)}\pgflineto{\pgfxy(2.6,3.9)}\pgflineto{\pgfxy(2.5,4)}\pgfclosepath\pgffillstroke
\color{black}
\pgfmoveto{\pgfxy(3,4.5)}\pgfpatharc{180}{270}{0.5cm}
\pgfstroke
\pgfmoveto{\pgfxy(3,4.5)}\pgflineto{\pgfxy(2.9,4.4)}\pgflineto{\pgfxy(3.1,4.4)}\pgflineto{\pgfxy(3,4.5)}\pgfclosepath\pgffillstroke

%Box (1,4)
\pgfsetlinewidth{1.2pt} 
\color{black}
\pgfmoveto{\pgfxy(3.5,4)}\pgfpatharc{90}{0}{0.5cm}
\pgfstroke
\pgfmoveto{\pgfxy(3.5,4)}\pgflineto{\pgfxy(3.6,4.1)}\pgflineto{\pgfxy(3.6,3.9)}\pgflineto{\pgfxy(3.5,4)}\pgfclosepath\pgffillstroke
\color{black}
\pgfmoveto{\pgfxy(4,4.5)}\pgfpatharc{180}{270}{0.5cm}
\pgfstroke
\pgfmoveto{\pgfxy(4,4.5)}\pgflineto{\pgfxy(3.9,4.4)}\pgflineto{\pgfxy(4.1,4.4)}\pgflineto{\pgfxy(4,4.5)}\pgfclosepath\pgffillstroke

%Box (2,1)
\pgfsetlinewidth{1.2pt} 
\color{black} 
\pgfxyline(1,2.5)(1,3.5)
\pgfmoveto{\pgfxy(0.9,3.4)}\pgflineto{\pgfxy(1,3.5)}\pgflineto{\pgfxy(1.1,3.4)}\pgflineto{\pgfxy(0.9,3.4)}\pgfclosepath\pgffillstroke
\color{black}
\pgfxyline(0.5,3)(1.5,3)
\pgfmoveto{\pgfxy(0.5,3)}\pgflineto{\pgfxy(0.6,3.1)}\pgflineto{\pgfxy(0.6,2.9)}\pgflineto{\pgfxy(0.5,3)}\pgfclosepath\pgffillstroke

%Box (2,2)
\pgfsetlinewidth{1.2pt} 
\color{black} 
\pgfxyline(2,2.5)(2,3.5)
\pgfmoveto{\pgfxy(1.9,3.4)}\pgflineto{\pgfxy(2,3.5)}\pgflineto{\pgfxy(2.1,3.4)}\pgflineto{\pgfxy(1.9,3.4)}\pgfclosepath\pgffillstroke
\color{black}
\pgfxyline(1.5,3)(2.5,3)
\pgfmoveto{\pgfxy(1.5,3)}\pgflineto{\pgfxy(1.6,3.1)}\pgflineto{\pgfxy(1.6,2.9)}\pgflineto{\pgfxy(1.5,3)}\pgfclosepath\pgffillstroke

%Box (2,3)
\pgfsetlinewidth{1.2pt} 
\color{black} 
\pgfxyline(3,2.5)(3,3.5)
\pgfmoveto{\pgfxy(2.9,3.4)}\pgflineto{\pgfxy(3,3.5)}\pgflineto{\pgfxy(3.1,3.4)}\pgflineto{\pgfxy(2.9,3.4)}\pgfclosepath\pgffillstroke
\color{black}
\pgfxyline(2.5,3)(3.5,3)
\pgfmoveto{\pgfxy(2.5,3)}\pgflineto{\pgfxy(2.6,3.1)}\pgflineto{\pgfxy(2.6,2.9)}\pgflineto{\pgfxy(2.5,3)}\pgfclosepath\pgffillstroke

%Box (2,4)
\pgfsetlinewidth{1.2pt} 
\color{black}
\pgfmoveto{\pgfxy(3.5,3)}\pgfpatharc{90}{0}{0.5cm}
\pgfstroke
\pgfmoveto{\pgfxy(3.5,3)}\pgflineto{\pgfxy(3.6,3.1)}\pgflineto{\pgfxy(3.6,2.9)}\pgflineto{\pgfxy(3.5,3)}\pgfclosepath\pgffillstroke
\color{black}
\pgfmoveto{\pgfxy(4,3.5)}\pgfpatharc{180}{270}{0.5cm}
\pgfstroke
\pgfmoveto{\pgfxy(4,3.5)}\pgflineto{\pgfxy(3.9,3.4)}\pgflineto{\pgfxy(4.1,3.4)}\pgflineto{\pgfxy(4,3.5)}\pgfclosepath\pgffillstroke

%Box (3,1)
\pgfsetlinewidth{1.2pt} 
\color{black}
\pgfmoveto{\pgfxy(0.5,2)}\pgfpatharc{90}{0}{0.5cm}
\pgfstroke
\pgfmoveto{\pgfxy(0.5,2)}\pgflineto{\pgfxy(0.6,2.1)}\pgflineto{\pgfxy(0.6,1.9)}\pgflineto{\pgfxy(0.5,2)}\pgfclosepath\pgffillstroke
\color{black}
\pgfmoveto{\pgfxy(1,2.5)}\pgfpatharc{180}{270}{0.5cm}
\pgfstroke
\pgfmoveto{\pgfxy(1,2.5)}\pgflineto{\pgfxy(0.9,2.4)}\pgflineto{\pgfxy(1.1,2.4)}\pgflineto{\pgfxy(1,2.5)}\pgfclosepath\pgffillstroke

%Box (3,2)
\pgfsetlinewidth{1.2pt} 
\color{black} 
\pgfxyline(2,1.5)(2,2.5)
\pgfmoveto{\pgfxy(1.9,2.4)}\pgflineto{\pgfxy(2,2.5)}\pgflineto{\pgfxy(2.1,2.4)}\pgflineto{\pgfxy(1.9,2.4)}\pgfclosepath\pgffillstroke
\color{black}
\pgfxyline(1.5,2)(2.5,2)
\pgfmoveto{\pgfxy(1.5,2)}\pgflineto{\pgfxy(1.6,2.1)}\pgflineto{\pgfxy(1.6,1.9)}\pgflineto{\pgfxy(1.5,2)}\pgfclosepath\pgffillstroke

%Box (3,3)
\pgfsetlinewidth{1.2pt} 
\color{black}
\pgfmoveto{\pgfxy(2.5,2)}\pgfpatharc{90}{0}{0.5cm}
\pgfstroke
\pgfmoveto{\pgfxy(2.5,2)}\pgflineto{\pgfxy(2.6,2.1)}\pgflineto{\pgfxy(2.6,1.9)}\pgflineto{\pgfxy(2.5,2)}\pgfclosepath\pgffillstroke
\color{black}
\pgfmoveto{\pgfxy(3,2.5)}\pgfpatharc{180}{270}{0.5cm}
\pgfstroke
\pgfmoveto{\pgfxy(3,2.5)}\pgflineto{\pgfxy(2.9,2.4)}\pgflineto{\pgfxy(3.1,2.4)}\pgflineto{\pgfxy(3,2.5)}\pgfclosepath\pgffillstroke

%Box (3,4)
\pgfsetlinewidth{1.2pt} 
\color{black}
\pgfmoveto{\pgfxy(3.5,2)}\pgfpatharc{90}{0}{0.5cm}
\pgfstroke
\pgfmoveto{\pgfxy(3.5,2)}\pgflineto{\pgfxy(3.6,2.1)}\pgflineto{\pgfxy(3.6,1.9)}\pgflineto{\pgfxy(3.5,2)}\pgfclosepath\pgffillstroke
\color{black}
\pgfmoveto{\pgfxy(4,2.5)}\pgfpatharc{180}{270}{0.5cm}
\pgfstroke
\pgfmoveto{\pgfxy(4,2.5)}\pgflineto{\pgfxy(3.9,2.4)}\pgflineto{\pgfxy(4.1,2.4)}\pgflineto{\pgfxy(4,2.5)}\pgfclosepath\pgffillstroke

%Box (4,1)
\pgfsetlinewidth{1.2pt} 
\color{black}
\pgfmoveto{\pgfxy(0.5,1)}\pgfpatharc{90}{0}{0.5cm}
\pgfstroke
\pgfmoveto{\pgfxy(0.5,1)}\pgflineto{\pgfxy(0.6,1.1)}\pgflineto{\pgfxy(0.6,0.9)}\pgflineto{\pgfxy(0.5,1)}\pgfclosepath\pgffillstroke
\color{black}
\pgfmoveto{\pgfxy(1,1.5)}\pgfpatharc{180}{270}{0.5cm}
\pgfstroke
\pgfmoveto{\pgfxy(1,1.5)}\pgflineto{\pgfxy(0.9,1.4)}\pgflineto{\pgfxy(1.1,1.4)}\pgflineto{\pgfxy(1,1.5)}\pgfclosepath\pgffillstroke

%Box (4,2)
\pgfsetlinewidth{1.2pt} 
\color{black}
\pgfmoveto{\pgfxy(1.5,1)}\pgfpatharc{90}{0}{0.5cm}
\pgfstroke
\pgfmoveto{\pgfxy(1.5,1)}\pgflineto{\pgfxy(1.6,1.1)}\pgflineto{\pgfxy(1.6,0.9)}\pgflineto{\pgfxy(1.5,1)}\pgfclosepath\pgffillstroke
\color{black}
\pgfmoveto{\pgfxy(2,1.5)}\pgfpatharc{180}{270}{0.5cm}
\pgfstroke
\pgfmoveto{\pgfxy(2,1.5)}\pgflineto{\pgfxy(1.9,1.4)}\pgflineto{\pgfxy(2.1,1.4)}\pgflineto{\pgfxy(2,1.5)}\pgfclosepath\pgffillstroke

%Box (4,3)
\pgfsetlinewidth{1.2pt} 
\color{black}
\pgfmoveto{\pgfxy(2.5,1)}\pgfpatharc{90}{0}{0.5cm}
\pgfstroke
\pgfmoveto{\pgfxy(2.5,1)}\pgflineto{\pgfxy(2.6,1.1)}\pgflineto{\pgfxy(2.6,0.9)}\pgflineto{\pgfxy(2.5,1)}\pgfclosepath\pgffillstroke
\color{black}
\pgfmoveto{\pgfxy(3,1.5)}\pgfpatharc{180}{270}{0.5cm}
\pgfstroke
\pgfmoveto{\pgfxy(3,1.5)}\pgflineto{\pgfxy(2.9,1.4)}\pgflineto{\pgfxy(3.1,1.4)}\pgflineto{\pgfxy(3,1.5)}\pgfclosepath\pgffillstroke

%Box (4,4)
\pgfsetlinewidth{1.2pt} 
\color{black}
\pgfmoveto{\pgfxy(3.5,1)}\pgfpatharc{90}{0}{0.5cm}
\pgfstroke
\pgfmoveto{\pgfxy(3.5,1)}\pgflineto{\pgfxy(3.6,1.1)}\pgflineto{\pgfxy(3.6,0.9)}\pgflineto{\pgfxy(3.5,1)}\pgfclosepath\pgffillstroke
\color{black}
\pgfmoveto{\pgfxy(4,1.5)}\pgfpatharc{180}{270}{0.5cm}
\pgfstroke
\pgfmoveto{\pgfxy(4,1.5)}\pgflineto{\pgfxy(3.9,1.4)}\pgflineto{\pgfxy(4.1,1.4)}\pgflineto{\pgfxy(4,1.5)}\pgfclosepath\pgffillstroke

%Label
\pgfputat{\pgfxy(0.2,1)}{\pgfnode{rectangle}{center}{\color{black}\small 1}{}{\pgfusepath{}}}
\pgfputat{\pgfxy(0.2,2)}{\pgfnode{rectangle}{center}{\color{black}\small 2}{}{\pgfusepath{}}}
\pgfputat{\pgfxy(0.2,3)}{\pgfnode{rectangle}{center}{\color{black}\small 3}{}{\pgfusepath{}}}
\pgfputat{\pgfxy(0.2,4)}{\pgfnode{rectangle}{center}{\color{black}\small 4}{}{\pgfusepath{}}}

\pgfputat{\pgfxy(4.8,1)}{\pgfnode{rectangle}{center}{\color{black}\small 1}{}{\pgfusepath{}}}
\pgfputat{\pgfxy(4.8,2)}{\pgfnode{rectangle}{center}{\color{black}\small 2}{}{\pgfusepath{}}}
\pgfputat{\pgfxy(4.8,3)}{\pgfnode{rectangle}{center}{\color{black}\small 3}{}{\pgfusepath{}}}
\pgfputat{\pgfxy(4.8,4)}{\pgfnode{rectangle}{center}{\color{black}\small 4}{}{\pgfusepath{}}}

\pgfputat{\pgfxy(1,0.2)}{\pgfnode{rectangle}{center}{\color{black}\small -4}{}{\pgfusepath{}}}
\pgfputat{\pgfxy(2,0.2)}{\pgfnode{rectangle}{center}{\color{black}\small -3}{}{\pgfusepath{}}}
\pgfputat{\pgfxy(3,0.2)}{\pgfnode{rectangle}{center}{\color{black}\small -2}{}{\pgfusepath{}}}
\pgfputat{\pgfxy(4,0.2)}{\pgfnode{rectangle}{center}{\color{black}\small -1}{}{\pgfusepath{}}}

\pgfputat{\pgfxy(1,4.8)}{\pgfnode{rectangle}{center}{\color{black}\small -4}{}{\pgfusepath{}}}
\pgfputat{\pgfxy(2,4.8)}{\pgfnode{rectangle}{center}{\color{black}\small -3}{}{\pgfusepath{}}}
\pgfputat{\pgfxy(3,4.8)}{\pgfnode{rectangle}{center}{\color{black}\small -2}{}{\pgfusepath{}}}
\pgfputat{\pgfxy(4,4.8)}{\pgfnode{rectangle}{center}{\color{black}\small -1}{}{\pgfusepath{}}}

\end{pgfpicture} 

\caption{Example of applying pipe-dreams to a symmetric Cauchon diagram}
\label{pipedreamsexample}
\end{center}
\end{figure}

The signed permutation $\tau_\Delta$ sends $i$ (on the bottom or right side of $\Delta$) to its image $\tau_\Delta(i)$ (on the top or left side of $\Delta$) by following the pipe starting at $i$. Note that when traversing a black square, we always go directly across the square. For the example in Figure~\ref{pipedreamsexample}, we find that $\tau_\Delta=(1\,\,-4)(-1\,\,4)(2\,\,3\,\,-2\,\,-3)$ (where we have written the permutation in the standard disjoint cycle notation). We omit the proof of the following lemma but it may be justified by arguments similar to those in~\cite{BCLqm}.

\begin{lem}\label{lem1}
For a Cauchon diagram $\Delta$, we have $\tau_\Delta=w^\Delta w^{-1}$.
\end{lem}

\subsection{A Basis for $\ker(w^\Delta+w)$} \label{types}
We now show how one may use cycles of $\tau_\Delta$ to calculate $\dim(\ker(w^\Delta+w))$. For a sequence $R=(a_1,a_2,\ldots,a_k)$, define $-R=(-a_1,-a_2,\ldots,-a_k)$. The cycle structure of $\tau_\Delta$ is easily seen to be a disjoint union of cycles that we group into the following types:

\begin{itemize}
\item[Type (a).] Pairs of fixed points $(i)(-i)$. (Such pairs occur if and only if $i$ is an all-black row in $\Delta$.)
\item[Type (b).] A pair of cycles of the form $(R_1,-R_2,R_3,\ldots,-R_{2m})(-R_1,R_2,\ldots,R_{2m})$, for some $m\geq 1$, where $\cup_{i=1}^{2m} R_i\subseteq [n]$ is a disjoint union, and, without loss of generality, each sequence $R_i$ of positive integers is monotone increasing.
\item[Type (c).] A single cycle of the form $(R_1,-R_2,\ldots,R_{2m-1},-R_1,R_2,\ldots,-R_{2m-1})$ for some $m\geq 1$, where $\cup_{i=1}^{2m-1} R_i\subseteq [n]$ is a disjoint union, and, without loss of generality, each sequence $R_i$ of positive integers is monotone increasing.
\end{itemize}

\begin{rem} \label{typerem}
It will be convenient to abuse pluralization by calling a pair of cycles of type (b), a ``cycle of type 2''.  Call a cycle of type  (b) or  (c) \emph{even} or \emph{odd} according to the parity of $\sum_i |R_i|$.
\end{rem}

Referring again to Figure~\ref{pipedreamsexample}, we see that $(1\,\,-4)(-1\,\,4)$ is an even cycle of type  (b) and $(2\,\,3\,\,-2\,\,-3)$ is an even cycle of type  (c).

\begin{thm}\label{basisthm}
Let $\Delta$ be a Cauchon diagram. There is a bijection between a basis for $\ker(w^\Delta+w)$ and the set that consists of the even type  (b) cycles together with the odd type (c) cycles in $\tau_\Delta$.
\end{thm}

\begin{proof}
Let $\tau=\tau_\Delta$. For a signed permutation $\sigma$, let $P_\sigma$ be the $n\times n$ matrix (indexed by $[n]$) defined by setting \begin{displaymath}
 P_\sigma[i,j] = \left\{\begin{array}{ll}
1 & \textnormal{ if $\sigma(i)=j$, } \\
-1 & \textnormal{ if $\sigma(i)=-j$,}\\
0 & \textnormal{ otherwise}.
\end{array} \right.
\end{displaymath}

It is straightforward to check that the homomorphism $\sigma\mapsto P^T_\sigma$ is an $n$-dimensional representation of $S_{2n}^B$. Combining this fact with Lemma~\ref{lem1} gives $$\ker(w^\Delta+w)\simeq \ker(1+w^\Delta w^{-1}) \simeq \ker (I+P_{\tau}).$$ Now, note that a vector $v\in\mathbb{R}^n$ is in $\ker (I+P_{\tau})$ if and only if for every $i\in[n]$, the $i$th component $v_i$ of $v$ satisifies \begin{eqnarray}v_i+P_\tau[i,\tau(i)]v_{|\tau(i)|}&=&0.\label{fact} \end{eqnarray} This simple fact allows one to construct a basis $B$ for $ \ker (I+P_{\tau})$ from the grouped cycles in the hypothesis as follows. 

First suppose $(R_1,-R_2,\ldots,R_{2m-1},-R_1,R_2,\ldots,-R_{2m-1})$ is a cycle of type  (c) and set $r=\sum_{i=1}^{2m-1} \left|R_i\right|$. Let $a$ be the first entry in $R_1$ and construct $v\in \mathbb{R}^n$ by setting $v_a = 1$ and, for $1\leq j\leq k-1$, $$v_{|\tau^j(a)|} = (-1)^j\prod_{\ell=1}^jP_\tau[\tau^{\ell-1}(a),\tau^\ell(a)].$$ 
In other words, $v_{|\tau^j(a)|}$ is $(-1)^j$ times the number of sign changes in $\tau$ between $a$ and $\tau^j(a)$. Set the remainder of the components in $v$ to zero. 

Using induction, it is easy to check that $v$ satisfies~(\ref{fact}) for $i=|\tau^j(a)|, 0\leq j\leq r-2$. Moreover, since $\tau^r(a)=a$ we have 
\begin{eqnarray}
v_{|\tau^{r-1}(a)|} + v_{|\tau^r(a)|} &= & v_{|\tau^{r-1}(a)|} + 1\nonumber \\
&=& (-1)^{r-1}\prod_{\ell=1}^{r-1}P_\tau[\tau^{\ell-1}(a),\tau^\ell(a)] + 1\nonumber\\
&=& (-1)^{r-1}(-1)^{\textnormal{(number of sign changes in $\tau$)}-1}+1\label{fact2}.
\end{eqnarray}
Since $\tau$ is of type  (c), the number of sign changes in $\tau$ is even. Thus (\ref{fact2}) is equal to zero if and only if $r$ is odd. This shows that $v\in B$ as required. 

The argument for a cycle of type (b) is similar. Since all cycles are disjoint, it follows that for any $a\in [n]$, $v_a\neq 0$ for exactly one $v\in B$. Hence the elements of $B$ are linearly independent. Finally, since any $w\in\ker (I+P_{\tau})$ satisfies (\ref{fact}), we must have that for any $a\in[n]$ the values $|w_{|\tau^j(a)|}|$ are equal for all $j$. It follows that $w$ is a linear combination of elements in $B$. Thus $B$ is a basis for $\ker (I+P_{\tau})$.\end{proof}

We conclude from Theorem \ref{thm: bcl}, Lemma \ref{lem1} and Theorem \ref{basisthm} the following result.

\begin{cor}
The dimension of the $\C{H}$-stratum associated to $\Delta$ is equal to the total number of type  (b) even cycles plus the total number of type  (c) odd cycles in $\tau_\Delta$.
\end{cor}

\section{Enumeration of $\C{H}$-strata with respect to dimension}

Before we begin, let us recall that the Stirling number of the second kind $S(n,j)$ counts the number of partitions of $[n]$ into exactly $j$ non-empty subsets. The following facts are well known (e.g., see~\cite{stanley}).
\begin{prop} \label{snprop}
If $n$ and $j$ are nonnegative integers, then
\begin{eqnarray}
\frac{1}{j!}(e^x-1)^j &=& \sum_{m=j}^\infty S(m,j) \frac{x^m}{m!};\label{expansion}\\
S(n,j) &=& \frac{1}{j!}\sum_{i=0}^{j}(-1)^{j-i}\binom{j}{i}i^n.\label{stir1} \label{bound}
\end{eqnarray}
\end{prop}

We now derive a formal power series $H(x, t)\in \mathbb{Q}(t)[[x]]$, where the coefficient of $x^n/n!$ in $H(x)$ is the polynomial $p_n(t)=\sum_{d=0}^n c_dt^d$, $c_d$ being the number of $d$-dimensional $\C{H}$-strata in $U_q[w_{max}^{\{1,\dots,n-1\}}]$.

We find $H$ by using the well-known exponential formula from enumerative combinatorics. Roughly speaking, the formula can be applied in the following type of situation. Suppose one fixes a family $\C{F}$ of finite, ordered subsets from some fixed set. We call the elements of $\C{F}$ ``components''. Suppose each component in $\C{F}$ has a notion of size (a positive integer). Furthermore, for a field $\mathbb{K}$ of characteristic zero, let $f:\mathbb{N}\rightarrow \mathbb{K}$, and call $f(n)$ the ``weight'' associated to the components in $\C{F}$ of size $n$. Finally let $D(x)=\sum_{n\geq 1} f(n) \frac{x^n}{n!}\in \mathbb{K}[[x]]$ be the exponential generating function for the weights of the components of size $n$ in $\C{F}$.

\begin{thm}[{The Exponential Formula~\cite[Corollary 5.1.6]{stanley}}]
We continue with the notation from the previous paragraph. Let $\C{K}$ be the family whose members consist of all finite sets of mutually disjoint components in $\C{F}$. Let $h:(\mathbb{N}\cup\{0\})\rightarrow \mathbb{K}$ be defined by $h(0)=1$ and $h(n)=\sum f(|a_1|)\cdots f(|a_k|)$, where the sum is over all set partitions $a_1,\ldots,a_k$ of $[n]$ with $a_i\neq\emptyset$ for all $i$. If $H(x)\in \mathbb{K}[[x]]$ is the exponential generating function for the $h(n)$, then $$H(x)=\exp(D(x)).$$
\end{thm}

In our present work, a ``component''  means a single cycle of type  (a),  (b) or  (c) as in Section~\ref{types}. The size of a component is the sum $\sum_i |R_i|$ (as in Remark~\ref{typerem}). By Theorem~\ref{basisthm}, we would like our weights to distinguish the parities of the sizes of components. So we take $\mathbb{K}=\mathbb{Q}(t)$ and let the weight of $n$ be the polynomial $a_nt+b_n$, where $a_n$ and $b_n$ are as follows. 

First note that the only Cauchon diagram $\Delta$ with $\tau_\Delta$ consisting of exactly one type (a) cycle is the $1\times 1$ black square.  On the other hand, the Cauchon diagram corresponding to the $1\times 1$ white square is a single odd cycle of type (c). So we set $a_1t+b_1=t+1. $

Next, observe that if $\tau_\Delta$ consists of exactly one cycle of type (b) or  (c), then its type is determined by the parity of $n$: $\tau_\Delta$ is of type (c) if and only if $n\geq 3$ is odd. Therefore, we tag with a $t$ the cases where $n$ and the cycle $\tau_\Delta$ have the same parity. Thus, if $n\geq 3$ is odd, then we set $a_n$ to be the number of $n\times n$ diagrams whose $\tau_\Delta$ consists of a single type $(c)$ cycle, and $b_n$ to be the number of $n\times n$ diagrams whose $\tau_\Delta$ consists of a single type $(b)$ cycle. If $n$ is even, we interchange types (b) and (c) in the previous sentence. 

So if $D(x)=D(x,t)$ is the exponential generating function with weights as just described, then using some elementary combinatorics concerning the possible sequences in a type (b) or (c) cycle, imply that $D(x,t)$ is given by

\begin{eqnarray*}
D(x,t) & = & x+\sum_{n=1}^\infty \left[\sum_{j=1}^n\left(\frac{1+(-1)^{n+j}}{2}\right)(j-1)!S(n,j)t\right.\\
& & \hspace{2cm}+ \left.\left(\frac{1-(-1)^{n+j}}{2}\right)(j-1)!S(n,j)\right]\frac{x^n}{n!}.
\end{eqnarray*}
Now we simplify using standard identities (including those from Proposition~\ref{snprop}): 

\begin{align}
&   D(x,t)&    \nonumber\\
&   =    x+\frac{t}{2}\left(\sum_{j=1}^\infty(j-1)!\left(\sum_{n=j}^\infty S(n,j)\frac{x^n}{n!}\right) + \sum_{j=1}^\infty (-1)^j (j-1)!\sum_{n=j}^\infty S(n,j) \frac{(-x)^n}{n!}\right)  \nonumber\\
 & \hspace{0.7cm} + \frac{1}{2}\left(\sum_{j=1}^\infty(j-1)!\left(\sum_{n=j}^\infty S(n,j)\frac{x^n}{n!}\right) - \sum_{j=1}^\infty (-1)^j (j-1)!\sum_{n=j}^\infty S(n,j) \frac{(-x)^n}{n!}\right)\nonumber \\
 &=  x+\frac{t}{2}\left(\sum_{j=1}^\infty (j-1)! \frac{(e^x-1)^j}{j!} + \sum_{j=1}^\infty (j-1)!\frac{(1-e^{-x})^j}{j!}\right)  \nonumber\\
 & \hspace{0.7cm}+\frac{1}{2}\left(\sum_{j=1}^\infty (j-1)! \frac{(e^x-1)^j}{j!} - \sum_{j=1}^\infty (j-1)!\frac{(1-e^{-x})^j}{j!}\right) \nonumber\\
 & = x+\frac{t}{2}\left(-\log(2-e^x) -\log(e^{-x})\right) + \frac{1}{2}\left(-\log(2-e^x) +\log(e^{-x})\right) \nonumber\\ 
&=  x+\log(2-e^{x})^{\frac{-1-t}{2}} + \frac{t-1}{2}x. \nonumber 
\end{align}

 Hence we get:
 \begin{equation}
 D(x,t) =  \log(2-e^{x})^{\frac{-1-t}{2}} + \frac{t+1}{2}x. \label{Dxt}
 \end{equation}
 
 From the above calculations, we obtain the following result.

\begin{thm}\label{thm:main}
If $H(x)=H(x,t)$ is the generating function whose $t^d\frac{x^n}{n!}$ coefficient is the number of $d$-dimensional $\C{H}$-strata in $U_q[w_{max}^{\{1,\dots,n-1\}}]$, then 
$$H(x,t)=\left(\frac{e^x}{2-e^x}\right)^{\frac{t+1}{2}}.$$
\end{thm}
\begin{proof}
By the exponential formula we have $H(x,t)=\exp(D(x,t))$. Therefore the result follows from (\ref{Dxt}).
\end{proof}

Note that the coefficient of $\frac{x^n}{n!}$ in $H(x,1)$ is the total number of $\C{H}$-strata in $U_q[w_{max}^{\{1,\dots,n-1\}}]$. In other words,   the coefficient of $\frac{x^n}{n!}$ in $H(x,1)$ is the total number of $\C{H}$-primes in $U_q[w_{max}^{\{1,\dots,n-1\}}]$. On the other hand, the coefficient of $\frac{x^n}{n!}$ in $H(x,0)$ is the number of primitive $\C{H}$-primes in $U_q[w_{max}^{\{1,\dots,n-1\}}]$ as it follows from the Stratification Theorem of Goodearl and Letzter that a stratum is 0-dimensional exactly when the associated $\C{H}$-prime is primitive \cite[II.8.4]{bg}. So we deduce from the previous theorem the following results.

\begin{cor} 
Let $H(x,t)$ be as in the statement of Theorem \ref{thm:main}.  Then:
\begin{enumerate}
\item The exponential generating function for the total number of $\C{H}$-primes in $U_q[w_{max}^{\{1,\dots,n-1\}}]$ is given by
\begin{eqnarray}
H(x,1) &=& \frac{e^x}{2-e^x}.\label{totalnumber}
\end{eqnarray}
\item The exponential generating function for the total number of $\C{H}$-invariant primitive ideals in $U_q[w_{max}^{\{1,\dots,n-1\}}]$ is given by
\begin{eqnarray}
H(x,0) & = & \left( \frac{e^x}{2-e^x} \right)^{\frac{1}{2}}\label{primnumber}
\end{eqnarray}
\end{enumerate}
\end{cor}

\section{The proportion of primitive $\C{H}$-primes}

In this section we show that as $n\rightarrow\infty$, the proportion of $\C{H}$-primes that are primitive is asymptotically zero  in the algebra $U_q[w_{max}^{\{1,\ldots ,n-1\}}]$. This fact, while perhaps expected, is not obvious, especially in view of the $m\times n$ quantum matrices case, where if $m$ is fixed and $n\rightarrow\infty$, the asymptotic proportion of $\C{H}$-primes has been recently shown to be positive~\cite{BCLqm}.

If $F(x)$ is a formal power series with coefficients in some ring, we let $[x^n]F(x)$ denote the coefficient of $x^n$ in $F(x)$.

We let $\prim(n)$ denote the number of primitive  $\C{H}$-primes in $U_q[w_{max}^{\{1,\ldots ,n-1\}}]$. Recall that $\prim(n)$ is also the number of 0-dimensional strata of $U_q[w_{max}^{\{1,\ldots ,n-1\}}]$. Then Equation (\ref{primnumber}) gives

$$\prim(n)/n! = [x^n] \exp(x/2)(2-e^x)^{-1/2}.$$

Using the Taylor expansion of $(2-e^x)^{-1/2} = (1-(e^x-1))^{-1/2}$, we see that the coefficients of the power series expansion of $(2-e^x)^{-1/2}$ are nonnegative.  Since $\exp(x/2)$ is coefficient-wise less than or equal to $\exp(x)$ and has nonnegative coefficients, we see

$$\prim(n)/n! = [x^n] \exp(x/2)(2-e^x)^{-1/2} \le [x^n] \exp(x)(2-e^x)^{-1/2}.$$

Recalling from (\ref{totalnumber}) that the exponential generating function for the total number of $\C{H}$-primes is given by 
$$\exp(x)/(2-e^x),$$
we note that to show that the proportion of primitive $\C{H}$-primes in $U_q[w_{max}^{\{1,\ldots ,n-1\}}]$ tends to zero, it is sufficient to show that 
$$[x^n]\exp(x) (2-e^x)^{- 1/2}={\rm o}\left([x^n]\exp(x)/(2-e^x)\right).$$  
We accomplish this via the following lemmas.
\begin{lem} Suppose that $A(x)=\sum_{n\ge 0} a_n x^n/n!$ and $B(x)=\sum_{n\ge 0} b_n x^n/n!$ are two exponential generating functions with positive coefficients.  If $a_n = {\rm o}(b_n)$ and $\log(b_n)/\log(n)\to \infty$ as $n\to \infty$ then $$[x^n]\exp(x)A(x)={\rm o}\left( [x^n] \exp(x)B(x) \right).$$
\label{lem: xx}
\end{lem}
\begin{proof} Let $\varepsilon>0$.  By assumption, there exists some natural number $d$ such that $a_n<\varepsilon b_n/2$ for $n\ge d$.  Therefore, for $n$ sufficiently large, we have
\begin{eqnarray*} 
[\frac{x^n}{n!}]\exp(x)A(x)&=& \sum_{j\le n} {n\choose j} a_j \\
&\le &    \sum_{d\le j\le n} \varepsilon {n\choose j} b_j/2 + {n\choose d}(a_0+\cdots +a_d).
\end{eqnarray*}
By assumption $b_n$ grows faster than any polynomial.  Since $a_0,\ldots, a_d$ are constant, we see that $${n\choose d}(a_0+\cdots +a_d)<\varepsilon b_n/2$$ for all $n$ sufficiently large.
Thus
$$[\frac{x^n}{n!}]\exp(x)A(x)\le \varepsilon [\frac{x^n}{n!}]\exp(x)B(x)$$ for all sufficiently large $n$.  The result follows.
\end{proof}
\begin{lem} Suppose that $\{a_n\}$ and $\{b_n\}$ are two sequences of positive numbers such that
$a_n = {\rm o}(b_n)$ and $\log(b_n)/n \to \infty$ as $n\to \infty$.  Then $$\sum_{j=1}^n a_j S(n,j) = {\rm o}\left( \sum_{j=1}^n b_j S(n,j) \right).$$
\label{lem: xy}
\end{lem}
\begin{proof} Let $\varepsilon>0$.  By assumption, there exists some natural number $d$ such that $a_n<\varepsilon b_n/2$ for $n\ge d$. 
Thus
$$\sum_{j=1}^n a_j S(n,j) < \sum_{j<d} a_j S(n,j) + \varepsilon \sum_{d\le j\le n} b_j S(n,j)/2.$$
Now since Formula~(\ref{bound}) easily gives $S(n,j)< (2j)^n$ and $b_n$ grows superexponentially, we have that for $n$ sufficiently large
$$\sum_{j<d} a_j S(n,j) < (a_1+\cdots +a_d)S(n,d) \leq \varepsilon b_nS(n,n)/2.$$
\end{proof}
\begin{cor} In $U_q[w_{max}^{\{1,\ldots ,n-1\}}]$, the proportion of $\C{H}$-primes that are primitive tends to $0$ as $n\to \infty$.
\end{cor}
\begin{proof} Recall from the discussion above Lemma \ref{lem: xx} that we just need to prove 
$$[x^n]\exp(x) (2-e^x)^{- 1/2}={\rm o}\left([x^n]\exp(x)/(2-e^x)\right).$$
Note that
\begin{eqnarray*}
(2-e^x)^{-1/2} &=& (1-(e^x-1))^{-1/2} \\
&=& \sum_{j\ge 0} {-1/2\choose j}(-1)^j (e^x-1)^j \\
&=& \sum_{j\ge 0} {2j \choose j} (e^x-1)^j/4^j.
\end{eqnarray*}
Thus for $n\ge 1$, we have
$$[\frac{x^n}{n!}](2-e^x)^{-1/2} = \sum_{1\le j\le n} S(n,j)j! {2j \choose j}/4^j.$$
Similarly,
$$[\frac{x^n}{n!}](2-e^x)^{-1} =  \sum_{1\le j\le n} S(n,j)j!.$$
Letting $a_j = j! {2j \choose j}/4^j$ and $b_j = j!$ and noting that 
${2j \choose j}/4^j \to 0$ as $j\to \infty$, we see that
$$[x^n](2-e^x)^{-1/2} = {\rm o}\left( [x^n](2-e^x)^{-1} \right)$$ as $n\to \infty$ by Lemma \ref{lem: xy}.
Since $$\sum_{1\le j\le n} S(n,j)j!$$ grows superpolynomially in $n$, we have that
$$[x^n]\exp(x) (2-e^x)^{- 1/2}={\rm o}\left([x^n]\exp(x)/(2-e^x)\right)$$ by Lemma \ref{lem: xx}.
The result follows.
\end{proof}

\providecommand{\bysame}{\leavevmode\hbox to3em{\hrulefill}\thinspace}
\providecommand{\MR}{\relax\ifhmode\unskip\space\fi MR }
% \MRhref is called by the amsart/book/proc definition of \MR.
\providecommand{\MRhref}[2]{%
  \href{http://www.ams.org/mathscinet-getitem?mr=#1}{#2}
}
\providecommand{\href}[2]{#2}

\end{document}